\documentclass[9pt,reqno,twoside,final]{amsart}
\usepackage{amsmath,amstext,amsthm,amsfonts,amssymb,amscd}

\renewcommand{\phi}{\varphi}
\renewcommand{\epsilon}{\varepsilon}
\renewcommand{\kappa}{\varkappa}       

\newtheorem{theorem}{Theorem}
\newtheorem{proposition}{Proposition}

\theoremstyle{definition}

\theoremstyle{remark}

\theoremstyle{definition}

\newtheorem*{rem*}{Remark}
\newtheorem*{acknow*}{Acknowledgements}
\newtheorem*{examples*}{Examples}

\theoremstyle{plain}

\newtheorem*{theorem*}{Theorem}

\newtheorem{question}{Question}
\newenvironment{proof-sketch}{\noindent{\bf Sketch of Proof}\hspace*{1em}}{\qed\bigskip}
\newenvironment{proof-idea}{\noindent{\bf Proof Idea}\hspace*{1em}}{\qed\bigskip}
\newenvironment{proof-of-lemma}[1]{\noindent{\bf Proof of Lemma #1}\hspace*{1em}}{\qed\bigskip}
\newenvironment{proof-of-prop}[1]{\noindent{\bf Proof of Proposition #1}\hspace*{1em}}{\qed\bigskip}
\newenvironment{proof-of-thm}[1]{\noindent{\bf Proof of Theorem #1.}\hspace*{1em}}{\qed\bigskip}
\newenvironment{proof-attempt}{\noindent{\bf Proof Attempt}\hspace*{1em}}{\qed\bigskip}

\title[A construction of a finitely presented semigroup]{A construction of a finitely presented semigroup containing nonnilpotent nil ideal}
\author{Ilya Ivanov-Pogodaev, Sergey Malev}
\address{Moscow Institute of Physics and Technology, Moscow, Russia }
\email{ivanov.pogodaev@gmail.com}
\address{School of mathematics, University of Edinburgh,
Edinburgh, UK; Department of Mathematics, Bar-Ilan University, Ramat Gan Israel}
\email{sergey.malev@ed.ac.uk; malevs@math.biu.ac.il}
\thanks{We would like to thank Agata Smoktunowicz and Alexei Kanel-Belov for interesting and fruitful discussions regarding this
paper}
\thanks{The second named author is partially supported by an Israeli Science Foundation grant number 1623/16}
\thanks{This research was supported by ERC Advanced grant Coimbra 320974.}

\thanks{This research was supported by Young Russian Mathematics award.}

\thanks{This research was supported by RFBR grant 14-01-00548.}

\begin{document}
\maketitle

\renewcommand{\baselinestretch}{1.0}
\renewcommand{\labelenumi}{[\theenumi]}

\begin{abstract}
 This work presents an example of a finitely presented semigroup $S$ containing an infinite nonnilpotent nil ideal $LS$,
 whose elements do not have a square 
(i.e. any word of the type $LXYYZ$ equals zero.)
 \end{abstract}

\section{Introduction}
In noncommutative algebra, ring theory, group and semigroup theory there were constructed various ``monsters'' 
giving counterexamples to some classical questions.
These counterexamples are defined by an infinite number of defining relations, and research effort is usually 
directed towards showing  
that these relations interact poorly and thus consequences (from relations) can be controlled.
Problems related to the construction of finitely presented objects with interesting properties were actively introduced by Latyshev.

A combinatorics of words questions in ring theory are raised in the monograph \cite{BBL}.

Recently several finitely presented objects were constructed.
In particular, a finitely presented semigroup with noninteger Gelfand-Kirillov dimension was constructed (see \cite{BI}).
A construction of finitely presented algebras with finite Gr\"obner basis with 
unsolvable problems of zero divisors and nilpotency was also provided (see \cite{IPM}).
All these results were achieved with realization of Turing Machine analogs with defining relations.

\medskip

A construction of a finitely presented infinite nil semigroup (see \cite{IPKB}) uses a different method:
semigroup elements are considered as a code of paths on a geometric complex, which has a set of special properties, 
in particular ellipticity and aperiodicity.
Using this method, relation is a cell of a complex, and transformation of the word is a changing of the path on complex, 
which saves its beginning, its end and its length.
This construction includes a lot of technical tests of the complex properties, and it increases the volume of the paper.

In the present paper we consider a weaker formulation: does there exist a finitely presented semigroup $S$ with zero, 
which has a letter $L$ in
alphabet, such that the ideal $LS$ is infinite and all the elements of it do not have a square 
(i.e. any word of the type $LXYYZ$ equals zero).

This formulation differs from the general question by existence of the ``fixed point'', i.e. in any analyzed word
we can use that it contains a letter $L$.
It allows us to provide a much simpler construction ($5$ pages instead of $160$).

\medskip

Note that, considering the problem of existence of finitely presented nil ring, we can ask an interim question:

\medskip

\begin{question}
Does theere exist a  finitely presented ring $R$ with zero, such that it has a letter $L$, the ideal $LR$ is infinite, 
and for any 
element $X\in R$ there exists an $n$ such that $LX^n=0$?
\end{question}

\section{Construction}

\begin{theorem}

There exists a finitely presented semigroup $H$, which has the following properties:

\begin{enumerate}

\item[(i)] There exists a nonnilpotent ideal $I=LH$, where $L$ is a letter in $H$;

\item[(ii)] If a word  $A\in H$ can be written as     $A=XYYZ$, where $X,\,Y,\,Z \in H$,
then $LA=0$.

\end{enumerate}
\end{theorem}

Take an alphabet $\Phi$:

 $$\{ L, \, M, \, P,\,Q,\, R,\, g,\,s_1,\, s_2,\, t_1,\,t_2,\,t_3 ,\,a_1,\,a_2,\,a_3,\,0\}.$$

Consider a semigroup $H$ with a zero element, generated by words in $\Phi$.
Our goal is to construct a finite number of defining relations, generating the required structure on $H$.

Consider the following set of relations:

 \begin{eqnarray}
   &
xL=xM=0,   \quad \text{where $x$ is any letter from $\Phi$}; \label{i1} \\ &
L=MPg,  \quad  \label{i2}\\  &
ga_i=a_ig, \quad i=1\dots 3 \label{i3} \\ &
gx=0,    \quad \text{where $x$ is any letter from $\Phi$ except} \ a_1, a_2, a_3; \label{i4} \\ &
a_iga_j=a_iRs_1Qa_j,    \quad i,j=1\dots 3; \label{i5}\\  &
t_ix=xt_i, \text{where $x\in \{R,\,a_1,\,a_2,\,a_3\}$, }
       i=1\dots 3; \label{i6}\\  &
Pa_it_i=a_iPs_1;           \label{i7} \\ &
Pa_jt_i=a_jPs_2,           \quad i,j=1\dots 3,\ i\ne j;   \label{i8} \\ &
s_ja_i=a_is_j,               \quad i=1\dots 3,\ j=1,2;  \label{i9}  \\ &
s_1R=Rs_1;                   \label{i10} \\ &
s_1Qa_i=t_ia_iQ,           \quad i=1\dots 3; \label{i11} \\ &
PRs_1=0;                    \label{i12} \\  &
s_2Ra_i=a_is_2R,             \quad i=1\dots 3; \label{i13} \\ &
s_2RQa_i=Rt_ia_iQ,         \quad i=1\dots 3. \label{i14}
\end{eqnarray}

\begin{proposition} \label{il1}
Any nonzero word $W\in H$ containing $L$ has a lexicographical form
 $W\equiv LA$, where $A$ is a word which consists of $a_1,\,a_2$ and $a_3$. By that we mean that $A$ belongs 
 to the subsemigroup generated by $a_1,\,a_2$ and $a_3$.
\end{proposition}

\begin{proof}
Let $W$ contain $L$. As we know, $xL=0$, for any letter $x$ (see relation \eqref{i1}), 
therefore $L$ can only be the first letter in the word $W$.
Assume $W=LU$. We use $L=MPg$ (see relation \eqref{i2}), thus $LU\equiv MPgU$.
Note that $ga_i=a_ig$ and $gx=0$ for $x\ne a_1,a_2,a_3$ (relations \eqref{i3} and \eqref{i4}).
Hence, if there is any letter except $a_1,\,a_2$ and $a_3$ in the word $U$, then $W$ simplifies to zero.
\end{proof}

\begin{proposition} \label{il2}
Let $X$ and $Y$ be any nonzero words. Then the word $LXY$ simplifies to either zero,
or the form $MPXRs_1QY$.
\end{proposition}

\begin{proof}
 If $X$ or $Y$ contains any letter except $a_i$, then according to
Proposition \ref{il1} a word $LXY$ equals zero. Let $X$ and $Y$ consist of letters $a_1,\,a_2,\,a_3$.
  Using $L=MPg$ (relation \eqref{i2}) we have $LXY\equiv MPgXY$.
Then using \eqref{i3} and \eqref{i5} we obtain $MPgXY\equiv MPXgY\equiv
MPXRs_1QY$.
\end{proof}

\begin{proposition} \label{il3}
Let $U$ be a word which consists of letters $a_1,\,a_2,\,a_3$.
Then $Pa_iURt_i\equiv a_iPURs_1$.

Moreover, if $i\ne j$ then $Pa_jURt_i\equiv a_jPUs_2R$.
\end{proposition}

\begin{proof}
According to relation \eqref{i6},  $Pa_iURt_i\equiv Pa_it_iUR$. Applying relation \eqref{i7},
we obtain $Pa_it_iUR \equiv a_iPs_1UR$. Therefore, according to \eqref{i9} and \eqref{i10},
we have $a_iPs_1UR\equiv a_iPURs_1$.

Assume $i\ne j$. According to relation \eqref{i6}, $Pa_jURt_i\equiv Pa_jt_iUR$. 
Hence, $Pa_jt_iUR \equiv a_jPs_2UR$ follows from \eqref{i8}. Therefore, according to  \eqref{i9},
we have $a_jPs_2UR\equiv a_jPUs_2R$.
\end{proof}

\begin{proposition} \label{il4}
Let $V$ be a word which consists of letters $a_1,\,a_2,\,a_3$.
Then $s_1VQa_i\equiv t_iVa_iQ$ and $s_2RVQa_i\equiv t_iVRa_iQ$.
\end{proposition}

\begin{proof}
According to relations \eqref{i9} and \eqref{i10}, $s_1VQa_i\equiv Vs_1Qa_i$.
Let us apply relation \eqref{i11},
thus  $Vs_1Qa_i \equiv Vt_ia_iQ$. According to \eqref{i6},
we have $Vt_ia_iQ\equiv t_iVa_iQ$.

Let us apply \eqref{i13}, as a consequence we have $s_2RVQa_i\equiv Vs_2RQa_i$.
According to \eqref{i14},
we have $Vs_2RQa_i \equiv VRt_ia_iQ$. Using \eqref{i6},
we have $VRt_ia_iQ\equiv t_iVRa_iQ$.
\end{proof}

\begin{proposition} \label{il5}
Let $X,\,V \,Z$ be words consist of letters $a_1,\,a_2,\,a_3$.
Then $PXVRZs_1QX\equiv XPVRs_1ZXQ$ and $PXs_2VRQX\equiv XPVRs_1XQ\equiv 0$.
\end{proposition}

\begin{proof}
Let us prove the first equivalence by induction on the length of $X$. 
For $X=a_i$ it follows from Proposition \ref{il3} and Proposition \ref{il4}. 
Let $X=a_iU$, i.e. $a_i$ is the first letter of $X$.
According to Proposition \ref{il4}, $PXVRZs_1QX=PXVRZ\underline{s_1Qa_i}U\equiv PXVRZ\underline{t_ia_iQ}U$.
Using \eqref{i6}, we have $PXVR\underline{Zt_i}a_iQU\equiv PXVR\underline{t_iZ}a_iQU$.
Applying Proposition \ref{il3}, we have
$PXVRt_iZa_iQU= \underline{Pa_iUVRt_i}Za_iQU\equiv \underline{a_iPUVRs_1}Za_iQU\equiv
a_iPUVRZa_is_1QU$. Thus we can consider
a word $PUVRZa_is_1QU$, and the induction hypothesis is true for it.

The second equivalence can be proved similarly, but on the first stage we use the second part of Proposition \ref{il4}.
\end{proof}

\begin{proposition} \label{il6}
Let $i,\, j, \, k \in \{1\dots 3\}$ and $i\ne j$. Then $Pga_ia_ja_k=a_iPga_ja_k$.
\end{proposition}

\begin{proof}
Using relations \eqref{i3} and \eqref{i5}, we have $Pga_ia_ja_k\equiv Pa_iRs_1Qa_ja_k$.
The relation \eqref{i11} gives us $Pa_iRs_1Qa_ja_k\equiv Pa_iRt_ja_jQa_k$.
Then, according to \eqref{i6} and \eqref{i8},
$\underline{Pa_iRt_j}a_jQa_k\equiv \underline{Pa_it_j}Ra_jQa_k \equiv
\underline{a_iPs_2}Ra_jQa_k$. Applying \eqref{i13}, we have $a_iP\underline{s_2Ra_j}Qa_k
\equiv a_iP\underline{a_js_2R}Qa_k$ and by \eqref{i14},
$a_iPa_j\underline{s_2RQa_k}\equiv a_iPa_j\underline{Rt_ka_kQ}$.
Now, according to \eqref{i5}, we have
$a_iP\underline{a_jRt_ka_kQ}\equiv a_iP\underline{a_jga_k}$. Therefore, according to $ga_j=a_jg$, $Pga_ia_ja_k=a_iPga_ja_k$.
\end{proof}

\begin{proposition} \label{il7}
Let $X,\,Y,\,Z$ be nonempty words. Then  $LXYYZ\equiv 0$.
\end{proposition}

\begin{proof}
According to Proposition \ref{il1}, we can assume, $X,\,Y,\,Z$ are words consist of letters $a_1,\,a_2,\,a_3$.

The relation \eqref{i1} gives us $LXYYZ\equiv MPgXYYZ$.
Applying Proposition \ref{il6} the required number of times, we obtain $MPgXYYZ\equiv MXPgYYZ$.

Now consider a subword $PgYY$. The relation Proposition \ref{il2} gives us $PgYY\equiv PYRs_1QY$.
Applying Proposition \ref{il5} (with empty words $V$ and $Z$) and \eqref{i12},
we obtain $PYRs_1QY\equiv YPRs_1YQ \equiv 0$.

\end{proof}

\begin{proposition} \label{il8}
Let a nonzero word $W$ contain letter $L$. Then, for a word $U$ equivalent to $W$,
there are three options:

\begin{enumerate}

\item[(i)] A word $U$ is of type $LA$, where $A$ is a word which consists of $a_1$, $a_2$, $a_3$;

\item[(ii)] A word $U$ is of type $MA_1PA_2gA_3$, where $A_1$, $A_2$, $A_3$ are
words that consist of $a_1$, $a_2$, $a_3$;

\item[(iii)] A word $U$ does not contain letters $L$ and $g$, however $U$ contains one letter (the first one) $M$,
one letter $P$, one $R$, one $Q$ and one letter from the set
$\{ t_1$, $t_2$, $t_3$, $s_1$, $s_2\}$. Moreover, letters $P$, $Q$, $R$ appear in this order.

\end{enumerate}
\end{proposition}

\begin{proof}

For any nonzero word $W$ let us introduce the following notions:

$I_0(W)$ -- number of letters $L$ and $M$ in the word $W$.

$I_1(W)$ -- number of letters  $L$ and $P$  in the word $W$.

$I_2(W)$ -- number of letters $L$, $g$, $R$  in the word $W$.

$I_3(W)$ -- number of letters $L$, $g$, $Q$ in the word $W$.

$I_4(W)$ -- number of letters $L$, $g$, $t_1$, $t_2$, $t_3$, $s_1$, $s_2$ in the word $W$.

Note that, the numbers $I_0$, $I_1$, $I_2$, $I_3$, $I_4$ are equal for any left and right part of any defining relation.
Therefore, they are invariant with respect to word equivalence. 
According to Proposition \ref{il1}, $W$ can be simplified to $LA$, where
$A$ is a product of letters $a_1,\,a_2,\,a_3$. Note
$I_0(LA)=I_1(LA)=I_2(LA)=I_3(LA)=I_4(LA)=1$. 
Therefore, for any nonzero word all the five invariants equal $1$.

If a word contains letter $L$, then, according to Proposition \ref{il1}, it does not contain letters $P$, $g$,
$R$, $Q$, $t_1$, $t_2$, $t_3$, $s_1$ and $s_2$. Therefore we have the option (i).

Assume the word does not contain the letter $L$. Thus it contains $M$ ($I_0=1$), moreover $M$ can be only the first from the left,
because there is only one 
relation with participation of $L$ and $M$: it is $L=MPg$. 
Furthermore,
there is only one letter $P$ in the word, because $I_1=1$. Consider two cases.

Let the word have letter $g$. Thus it is unique, and the word does not contain
$R$, $Q$, $t_1$, $t_2$, $t_3$, $s_1$ and $s_2$ because $I_2=I_3=I_4=1$.
 Thereby, the word is of type
$MA_1PA_2gA_3$, where $A_1$, $A_2$ and $A_3$ are words containing letters $a_1$, $a_2$ and $a_3$. Therefore, we have the option (ii).

Now assume the word does not contain letter $g$. Since $I_1=I_2=I_3=I_4=1$, the word has a unique letter
 $P$, unique letter $R$, unique letter $Q$ and a unique letter from the set
$\{ t_1$, $t_2$, $t_3$, $s_1$, $s_2\}$. According to Proposition \ref{il1} all words can be simplified to the type $LA$, 
which can be simplified to the type containing letters $P$, $Q$ and $R$ in this type. 
Moreover, there is no relation which can change this order.
Therefore, we have the option (iii)
\end{proof}

\begin{proposition} \label{il9}
Let  $U$ be a word containing letters $a_1$, $a_2$, $a_3$, and not containing squares of words.  Then  $LU\ne 0$.
\end{proposition}

\begin{proof}
Assume that $LU\equiv 0$. 
Thus there exists a chain of equal words, beginning with $LU$ and finishing with zero. 
Any transformation in this chain uses some defining relation. The last transformation is one of four relations:
$xL=0$, $gx=0$, $t_ia_jQ=0$, $PRs_1=0$. 
Let us consider all of them and check that none of them can happen.

The relation $xL=0$ cannot happen, because if the letter $L$ is contained in the word, it should be the first left
(if not, the first left is $M$).

The relation $gx=0$ cannot happen, since if the word contains  $g$, according to Proposition \ref{il8}
it will be of type $MA_1PA_2gA_3$, where $A_3$ consists of $a_1$, $a_2$, $a_3$.

Consider the relation $t_ia_jQ=0$. Note, any word equivalent to $LA$ (where $A$
consists of $a_1$, $a_2$, $a_3$) satisfies the following property: if
$s_1$ appears in the word, the last letter $a$ from the subword between  $R$ and $Q$
coninsides with the letter, closest from the left to $P$.  
The letter $Q$ does not appear in the beginning of the word. It can appear only by relation \eqref{i5}: $a_iga_j=a_iRs_1Qa_j$, 
i.e. when there is no letter $a$ between $R$ and $Q$. 
Assume we have switched to some equivalent word, which contains $s_1$.
It is easy to see that during this switch we had to use relations
$Pa_it_i=a_iPs_1$ and $s_1Qa_i=t_ia_iQ$ alternately and equal number of times. (Similarly
with relations $Pa_jt_i=a_jPs_2$ and $s_2RQa_i=Rt_ia_iQ$). A last letter $a$ of the subword between  $R$ and $Q$ 
coincides with the letter closest from the left to $P$. 
Thereby, if $s_1$ is contained in the word, the last letter $a$ from the subword between $R$ and $Q$ 
coincides with the letter closest from the left to $P$. 

Assume we have a word $MA_0PA_1RA_2t_ia_jQA_3$.

Consider relation $PRs_1=0$.
Assume there exists a chain of equivalent transformations, from the word $A_0PA_1Rs_1QA_2$ to the word $B_0PRs_1B_1QB_2$,
where words $A_i$ and $B_i$ consist
of $a_1$, $a_2$, $a_3$, moreover $A_0A_1A_2=B_0B_1B_2$ is a lexicographical equality.
We shall call $d(X,Y)$ {\it a distance between letters $X$ and $Y$}, a number of letters $a_1$, $a_2$, $a_3$ 
between $X$ and $Y$.
Note, 
``$d(P,Q) +$ the number of letters $s_1$ or $s_2$ in the word'' \ 
is invariant in the chain of equivalent transformations
(in other transformations there are no words containing $g$).
Furthermore, note that with equivalent transformation we obtain one of $t_1$, $t_2$, $t_3$ from letters $s_1$ and $s_2$,
and one of $s_1$ and $s_2$ from $t_1$, $t_2$ and $t_3$.
Thus, words in the general chain of equivalence can be divided to alternating sets
in such a way that
in one set all words contain $t_i$, in the next  - $s_1$ or $s_2$, and so on. 
Consider transformation, when we go away from the chain, i.e. we have a first word of the next chain as a result.
After this transformation all the words will not have any of $s_i$ except for $s_1$.
This transformation finishes a chain and thus,
it has relations containing $s_2$ only from one side.
That is, $s_2$ transforms in one of three letters: $t_1$, $t_2$ or $t_3$. There are only two relations of this type:
$a_jPs_2=Pa_jt_i$ and $s_2RQa_i=Rt_ia_iQ$.
Assume a transformation, which finishes a chain, where all words contain $s_2$, 
occured by relation $a_jPs_2=Pa_jt_i$. Thus, the first word of the next chain 
(all words of which contain $t_i$) 
is of type
$A_0Pa_jt_iA_1RA_2QA_3$. 
We considered the last chain with $s_2$, therefore all subsequent chains
will have either $s_1$, or $t_i$. 
That is, all subsequent transformations between chains are implemented by relations $s_1Qa_i=t_ia_iQ$ and $Pa_it_i=a_iPs_1$.
It is easy to see that applying these relations to the word $A_0Pa_jt_iA_1RA_2QA_3$,
the distance between $P$ and $R$ cannot decrease, because $j\ne i$.

Assume a transformation, finishing a chain, occured with the relation $s_2RQa_i=Rt_ia_iQ$.
Then, the first word of the next chain (all words of which contain $t_i$) is of type
$A_0PA_1Rt_ia_iQA_3$. 

After that no defining relation containing 
$s_2$ will apply, hence all subsequent transformations between chains occur with relations $s_1Qa_i=t_ia_iQ$ and $Pa_it_i=a_iPs_1$.
Therefore further we will have the following rule: if a word has $s_1$, then letters $a_i$
closest from the left to $P$ and $Q$ coincide, and if a word has $t_j$, then
$a_i$ closest from the left to $Q$ equals $a_j$. Moreover, note that
$R$ does not change its position in the word.
Assume that at some moment we have obtained a subword  $PRs_1$. 
Thus, we have switched from the word $A_0PA_1Rt_ia_iQA_3$ to the word $A_0A_1PRs_1a_i\widetilde{A_3}QA_4$.
In each step to the right from $P$ and $Q$ we had the same letter $a_i$.
Therefore, words $A_1$ and $a_i\widetilde{A_3}$ lexicographically equal, i.e. from the very beginning, the word had a square subword.
A contradiction.

Thus, it is impossible to obtain a word which equals zero.

\end{proof}

\pagestyle{headings}

\end{document}